\documentclass{article}

\usepackage[utf8]{inputenc}
\usepackage[english]{babel}
\usepackage{amsmath,amssymb,epsfig,amsthm}
\usepackage{graphicx}
\usepackage{comment}
\usepackage{array}
\usepackage{algorithm,algpseudocode}
\usepackage{enumerate}
\usepackage{algpseudocode}
\usepackage{xurl}
\usepackage{xargs}     
\usepackage{subcaption}
\usepackage{mwe}
\usepackage[pdftex,dvipsnames]{xcolor}  
\usepackage[colorinlistoftodos,prependcaption,textsize=tiny]{todonotes}
\usepackage{appendix}
\usepackage{textalpha}

\newcommandx{\unsure}[2][1=]{\todo[linecolor=red,backgroundcolor=red!25,bordercolor=red,#1]{#2}}
\newcommandx{\change}[2][1=]{\todo[linecolor=blue,backgroundcolor=blue!25,bordercolor=blue,#1]{#2}}
\newcommandx{\info}[2][1=]{\todo[linecolor=OliveGreen,backgroundcolor=OliveGreen!25,bordercolor=OliveGreen,#1]{#2}}
\newcommandx{\improvement}[2][1=]{\todo[linecolor=Plum,backgroundcolor=Plum!25,bordercolor=Plum,#1]{#2}}
\newcommandx{\thiswillnotshow}[2][1=]{\todo[disable,#1]{#2}}

\allowdisplaybreaks[1]

\newtheorem{theorem}{Theorem}

\newtheorem{assumption}{Assumption}

\newtheorem{condition}{Condition}

\newtheorem{corollary}{Corollary}

\newtheorem{definition}{Definition}

\newtheorem{lemma}{Lemma}

\newtheorem{remark}{Remark}

\numberwithin{equation}{section}

\DeclareMathOperator{\diag}{diag}

\DeclareMathOperator{\Var}{Var}

\newcommand{\calB}{\ensuremath{\mathcal{B}}}
\newcommand{\calC}{\ensuremath{\mathcal{C}}}

\newcommand{\calH}{\ensuremath{\mathcal{H}}}

\newcommand{\calG}{\ensuremath{\mathcal{G}}}

\newcommand{\calM}{\ensuremath{\mathcal{M}}}
\newcommand{\calN}{\ensuremath{\mathcal{N}}}

\newcommand{\calE}{\ensuremath{\mathcal{E}}}

\newcommand{\calL}{\ensuremath{\mathcal{L}}}

\newcommand{\norm}[1]{\left\|{#1}\right\|}
\newcommand{\abs}[1]{\left|{#1}\right|}

\newcommand{\expec}{\ensuremath{\mathbb{E}}}

\newcommand{\prob}{\ensuremath{\mathbb{P}}}

\definecolor{asparagus}{rgb}{0.53, 0.66, 0.42}

\newcommand{\indic}{\ensuremath{\mathbf{1}}} 

\newcommand{\R}{\ensuremath{\mathbb{R}}}

\newcommand{\spec}{\texttt{Spec}}
\newcommand{\specc}{\texttt{AdaSpec}}

\newcommand{\rev}[1]{\textcolor{black}{#1}}

\usepackage[round]{natbib}
\usepackage[margin=1in]{geometry} 
\begin{document}

\title{Strong Consistency Guarantees for Clustering High-Dimensional Bipartite Graphs with the Spectral Method}
\author{Guillaume Braun\footnotemark[1]\\ \texttt{guillaume.braun@riken.jp} }
\date{\today}

\renewcommand{\thefootnote}{\fnsymbol{footnote}}
\footnotetext[1]{RIKEN Center for Advanced Intelligence Project }

\maketitle

\begin{abstract}
    In this work, we focus on the Bipartite Stochastic Block Model (BiSBM), a popular model for bipartite graphs with a community structure. We consider the high dimensional setting where the number $n_1$ of type I nodes is far smaller than the number $n_2$ of type II nodes. The recent work of \cite{braun2022minimax} established a sufficient and necessary condition on the sparsity level $p_{max}$ of the bipartite graph to be able to recover the latent partition of type I nodes. They proposed an iterative method that extends the one proposed by \cite{Ndaoud2021ImprovedCA} to achieve this goal. Their method requires a good enough initialization, usually obtained by a spectral method, but empirical results showed that the refinement algorithm doesn't improve much the performance of the spectral method. This suggests that the spectral achieves exact recovery in the same regime as the refinement method. We show that it is indeed the case by providing  new entrywise bounds on the eigenvectors of the similarity matrix used by the spectral method. Our analysis extend the framework of \cite{lihua20} that only applies to symmetric matrices with limited dependencies. As an important technical step, we also derive an improved concentration inequality for similarity matrices. 
\end{abstract}

\section{Introduction}
Bipartite graphs are a convenient way to represent the relationships between objects of two different types. One can find examples of applications in many fields such as e-commerce with customers and products \cite{customer-product}, finance with investors and assets \cite{financial-bipartite}, and biology with plants of pollinators networks \cite{plants-pollinators}. These networks are often large, and sparse. Moreover, the number of type I and type II nodes can be quite different. 

To extract relevant information from these networks one often relies on clustering methods. Amongst them, spectral clustering (SC) is one of the most popular approaches due to its efficiency in terms of computational complexity and statistical accuracy. However, the existing consistency guarantees for SC are often weak or require a sub-optimal sparsity level, and do not fully explain the performance of SC, as observed experimentally in \cite{braun2022minimax} and \cite{Ndaoud2021ImprovedCA}.

In this work, we fill this gap by showing that the SC achieves exact recovery under the BiSBM, an asymmetric extension of the Stochastic Block Model (SBM) commonly used to evaluate the performance of the clustering algorithm for bipartite graphs. Besides, we show that SC is optimal in the sense that it achieves exact recovery whenever $n_1n_2p_{max}^2\gtrsim \log n_1$, the optimal sparsity regime. We leave as future work the characterization of the precise constant necessary for exact recovery. 

\subsection{Main contributions} Our main contributions are summarized below.
\begin{itemize}
    \item We show that the spectral method achieves exact recovery of the rows partition whenever $n_1n_2p_{max}^2\gtrsim \log n_1$ and is hence optimal. To do that, we extend to similarity matrices the entrywise concentration bounds for eigenvectors obtained by \cite{lihua20} for matrices with independent entries, or limited dependencies. 
    \item Our analysis applies to rank deficient connectivity matrix. It allows for the partially remove of the `` spectral gap condition '' -- a common condition in the analysis of spectral methods that requires that the matrices of interest satisfy some rank condition to ensure that there is a spectral gap -- as in the recent work of \cite{loffler21,zhang2022leaveoneout}.
    \item Central to our proof is an improved concentration bound for similarity matrices. We derive this result by adapting the combinatorial argument of \cite{freige05} used to show the concentration of adjacency matrices sampled from the generalized Erdös-Renyi model.
\end{itemize}

\subsection{Related work}
\paragraph{Bipartite graphs and spectral clustering.} The recent work of \cite{braun2022minimax} confirmed the conjecture of \cite{Ndaoud2021ImprovedCA} that $n_1n_2p_{max}^2 \gtrsim \log n_1$ is a necessary and sufficient condition for exact recovery of the rows partition under the high-dimensional BiSBM where $n_1\ll n_2$. This threshold can be achieved by generalized power methods proposed in the aforementioned articles. However, existing strong consistency guarantees for SC requires stronger assumption. For example, when specialized to the setting of \cite{Ndaoud2021ImprovedCA} (a special case of our more general model), the result of \cite{sub_est21} holds only when the sparsity level satisfies $n_1n_2p_{max}^2\gtrsim \log^2 n_2$. When $n_1n_2p_{max}^2 \gtrsim \log n_1$, SC is only guaranteed to achieve weak consistency \cite{braun2022minimax}. The work of \cite{florescu16} also showed that when $n_1n_2p_{max}^2\gtrsim 1$, one can recover a proportion of the type I nodes labels by a SBM reduction, but this is the weakest existing recovery guarantee and we are focusing on exact recovery. The recent work of \cite{zhang2022leaveoneout} also proposed an improved analysis of the spectral method for asymmetric matrices with independent entries, but their bound becomes trivial in the high-dimensional regime $n_1\ll n_2$ we are interested in. 

\paragraph{Entrywise concentration bounds for eigenvectors.} In recent years, spectral algorithms have been shown to successfully achieve exact recovery in various community detection tasks under various settings such as, e.g., the SBM \cite{Abbe2020ENTRYWISEEA}, the Contextual SBM \cite{abbe2020ellp}, the Censored Block Model \cite{Dhara2022SpectralRO}, Hierarchical SBM \cite{Lei2020ConsistencyOS} and uniform Hypergraph SBM \cite{gaudio2022community}. Spectral methods have also been used in other estimation problems such as group synchronization \cite{tyagi21}, ranking \cite{SpecRank2019}, or planted subgraph detection \cite{dhara2022spectral}. To prove these results, one generally needs to obtain entrywise eigenvector concentration bounds. In this work, we will follow the framework developed by \cite{lihua20} that combines techniques used to obtain deterministic perturbation bounds \cite{Fan2016AnE, cape2019b,sun2020} with techniques that rely on some stochastic properties of the noise \cite{Abbe2020ENTRYWISEEA,cape2019a,eldridge18a}.  

\subsection{Notations}
We use lowercase letters ($\epsilon, a, b, \ldots$) to denote scalars and vectors, except for universal constants that will be denoted by $c_1, c_2, \ldots$ for lower bounds, and $C_1, C_2, \ldots $ for upper bounds and some random variables. We will sometimes use the notation $a_n\lesssim b_n$ (or $a_n\gtrsim b_n$ ) for sequences $(a_n)_{n \geq 1}$ and $(b_n)_{n \geq 1}$ if there is a constant $C>0$ such that $a_n \leq C b_n$ (resp. $a_n \geq C b_n$) for all $n$. If the inequalities only hold for $n$ large enough, we will use the notation $a_n=O(b_n)$ (resp. $a_n=\Omega(b_n)$).  If $a_n \lesssim b_n$ (resp. $a_n=O(b_n)$) and $a_n \gtrsim b_n$ (resp. $a_n=\Omega(b_n)$), then we write $a_n \asymp b_n$ (resp. $a_n=\Theta(b_n)$).  

Matrices will be denoted by uppercase letters. The $i$-th row of a matrix $A$ will be denoted as $A_{i:}$. The column $j$ of $A$ will be denoted by $A_{:j}$, and the $(i,j)$th entry by $A_{ij}$. The transpose of $A$ is denoted by $A^\top$ and $A_{:j}^\top$ corresponds to the $j$th row of $A^\top$ by convention. $I_k$ denotes the $k\times k$ identity matrix. For matrices, we use $|| . ||$ and $||.||_F$ respectively denote the spectral norm (or Euclidean norm in the case of vectors) and the Frobenius norm. 
\section{Model and algorithm description}
\subsection{The Bipartite Stochastic Block Model (BiSBM)}
\label{subsec:stat_frame}

The BiSBM is a direct adaption of the SBM \cite{HOLLAND1983109} to bipartite graphs.  The model depends on the following parameters.
\begin{itemize}
    \item A set of nodes of type I, $\calN_1 =[n_1]$, and a set of nodes of type II, $\calN_2 =[n_2]$.
    
    \item A partition of $\calN_1$ into $K$ communities $\calC_1,\ldots, \calC_K$ 
    and a partition of $\calN_2$ into $L$ communities $\calC_1',\ldots, \calC'_L$. 
    
    \item Membership matrices $Z_1 \in \calM_{n_1,K}$ and $Z_2 \in \calM_{n_2,L}$ where $\calM_{n,K}$ denotes the class of membership matrices with $n$ nodes and $K$ communities.  Each membership matrix $Z_1\in \calM_{n_1,K}$ (resp. $Z_2\in \calM_{n_2,L}$) can be associated bijectively with a partition function $z:[n]\to [K]$ (resp. $z':[n]\to [L]$) such that $z(i)=z_i=k$  where $k$ is the unique column index satisfying $(Z_1)_{ik}=1$ (resp. $(Z_2)_{ik}=1$ ).
    
    \item A connectivity matrix of probabilities between communities \[\Pi=(\pi_{k k'})_{k\in [K], k'\in [L]} \in [0,1]^{K \times L}.\]
\end{itemize}
Let us write \[P=(p_{ij})_{i,j \in [n] }:=Z_1\Pi (Z_2)^\top \in [0,1]^{n_1\times n_2}.\] A graph $\calG$ is distributed according to BiSBM$(Z_1, Z_2, \Pi)$ if the entries of the corresponding bipartite adjacency matrix $A$ are generated by 
\[ 
A_{ij}  \overset{\text{ind.}}{\sim} \mathcal{B}(p_{ij}), \quad i \in [n_1], \ j \in [n_2],
\] 
where $\calB(p)$ denotes a Bernoulli distribution with parameter $p$. Hence the probability that two nodes are connected depends only on their community memberships. The sparsity level of the graph is denoted by $p_{max} = \max_{i,j} p_{ij}$. We make the following assumptions on the model. 
\begin{assumption}[Approximately balanced communities]\label{ass:balanced_part}
The communities  $\calC_1,\ldots, \calC_K$, 
(resp. $\calC_1',\ldots, \calC'_L$) are approximately balanced, i.e., there exists a constant $\alpha \geq 1$ such that for all $k\in [K]$ and $l\in [L]$ we have 
\[ 
\frac{n_1}{\alpha K}\leq \abs{\calC_k} \leq \frac{\alpha n_1}{K} \text{ and } \frac{n_2}{\alpha L}\leq \abs{\calC'_l} \leq \frac{\alpha n_2}{L}.
\] 
\end{assumption}
We will consider throughout this work the parameters $\alpha, K$ and $L$ as constants. We won't keep track in the stated bounds of the dependencies in these parameters.

We will rely on the following assumption to ensure that the communities are well separated.
\begin{assumption}[Communities are well separated]\label{ass:eig_low_bd}
 Let $U\Lambda U^\top$ be the spectral decomposition of $PP^\top$. All the communities are well separated if the following assumptions are satisfied.\begin{enumerate}
     \item The smallest non zero eigenvalue of $\Pi\Pi^\top$, denoted by $\lambda_{min}(\Pi\Pi^\top)$, satisfies $\lambda_{min}(\Pi\Pi^\top) \gtrsim p_{\max}^2$.
     \item For all $i, j\in [n_1]$ such that $z_i\neq z_j$ we have $\norm{U_{i:}-U_{j:}}\geq \frac{c_1}{\sqrt{n}}$.
 \end{enumerate}  
\end{assumption}
\begin{remark}
This assumption doesn't require that $\Pi \Pi^\top$ is full rank contrary to classical assumptions used for analyzing spectral clustering. For example, consider the setting where $K=2=L$, the communities are exactly balanced and
\[ \Pi \Pi^\top =
\begin{pmatrix}
p & cp \\
cp & c^2p
\end{pmatrix}
\]
where $p$ is the sparsity parameter and $c>0$ is a constant. Observe that \[ PP^\top= \frac{n_2}{2}Z_1\Pi \Pi^\top Z_1^\top = \frac{n_1n_2}{4} W\Pi \Pi^\top W^\top \] where $W=\sqrt{\frac{2}{n_1}}$ has orthonormal columns. The SVD decomposition of $\Pi \Pi^\top$ is given by $cpVV^\top$ where $V=(\frac{c}{\sqrt{1+c^2}},\frac{1}{\sqrt{1+c^2}})^\top$. Hence, $U=WV$ and for $i\in \calC_1$ and $j\in \calC_2$ we have \[ \norm{U_{i:}-U_{j:}}\geq \frac{|1-c|}{\sqrt{n_1}}.\] 
\end{remark}

 The quality of the clustering is evaluated through the \textbf{misclustering rate} $r$ defined by 
 \begin{equation}
 \label{eq:def_misclust}   
 r(\hat{z},z)=\frac{1}{n}\min _{\pi \in \mathfrak{S}}\sum_{i\in [n]} \indic_{\lbrace \hat{z}(i)\neq \pi(z(i))\rbrace},\end{equation}
 where $\mathfrak{S}$ denotes the set of permutations on $[K]$. We say that an estimator $\hat{z}$ achieves \textbf{exact recovery} if $r(\hat{z},z)=0$ with probability $1-o(1)$ as $n$ tends to infinity. It achieves \textbf{weak consistency} (or almost full recovery) if $\prob(r(\hat{Z},Z)=o(1))=1-o(1)$ as $n$ tends to infinity.  A more complete overview of the different types of consistency and the sparsity regimes where they occur can be found in \cite{AbbeSBM}.

\subsection{Algorithm description}
\label{subsec:algo}
In the high-dimensional and sparse setting where $n_1\ll n_2$ and $n_1n_2p_{max}^2$ is of order $\log n_1$, there is no hope to recover the columns partition $Z_2$. So, it is natural to form the similarity matrix $AA^\top$ and compute the top-$K$ eigenspace of this similarity matrix. Unfortunately, the diagonal elements of $AA^\top$ create an important bias ($(AA^\top)_{ii}$ is typically of order $n_2p_{max}$ while the diagonal entries of corresponding population similarity matrix are of order $n_2p_{max}^2$). To avoid this issue, one can remove the diagonal of $AA^\top$ and obtain a matrix $B$. In this work, we consider a slightly different variant of the spectral methods proposed by \cite{braun2022minimax, Ndaoud2021ImprovedCA,florescu16}. See Algorithm \ref{alg:spec} for a complete description of the method. 

\begin{algorithm}[hbt!]
\caption{Spectral method on $\calH(AA^\top)$ (\spec)}\label{alg:spec}
\begin{flushleft}
        \textbf{Input:} The number of communities $K$, the rank $r$ of $\Pi \Pi^\top$ and the adjacency matrix $A$.

        \begin{algorithmic}[1]
        \State Form the diagonal hollowed Gram matrix $B:=\calH(AA^\top) $ where $\calH(X)=X-\diag(X)$.
       \State Compute the matrix $U\in \R^{n_1\times r}$ whose columns correspond to the top $r$-eigenvectors of $B$.
       \State Apply approximate $(1+2/e+\epsilon)$ approximate $\texttt{k-medians}$ on the rows of $U$ and obtain a partition $z^{(0)}$ of $[n_1]$ into $K$ communities.
       \end{algorithmic}
 \textbf{Output:} A partition of the nodes $z^{(0)}$.\end{flushleft}
\end{algorithm}

When the rank of $\Pi \Pi^\top$ is not known, we propose \specc\, (see Algorithm \ref{alg:spec2}), an adaptive version of Algorithm \ref{alg:spec}.

\begin{algorithm}[hbt!]
\caption{Adaptive spectral method on $\calH(AA^\top)$ (\specc)}\label{alg:spec2}
\begin{flushleft}
        \textbf{Input:} The number of communities $K$, a threshold $T>0$, and the adjacency matrix $A$.

        \begin{algorithmic}[1]
        \State Form the diagonal hollowed Gram matrix $B:=\calH(AA^\top) $ where $\calH(X)=X-\diag(X)$.
       \State Let $\hat{r}\in [K]$ be the largest index such that the difference between two consecutive eigenvalues are larger than some threshold $T$ \[\hat{r}:=\arg \max \lbrace r\in [K]:  \lambda_r(B)-\lambda_{r+1}(B)>T \rbrace .\]
       \State Compute the matrix $U\in \R^{n_1\times r}$ whose columns correspond to the top $r$-eigenvectors of $B$.
       \State Apply approximate $(1+2/e+\epsilon)$ approximate $\texttt{k-medians}$ on the rows of $U$ and obtain a partition $z^{(0)}$ of $[n_1]$ into $K$ communities.
       \end{algorithmic}
 \textbf{Output:} A partition of the nodes $z^{(0)}$.\end{flushleft}
\end{algorithm}

\paragraph{Computational complexity.} The cost for computing $B$ is $O(n_1\texttt{nnz}(A))$  
\renewcommand{\thefootnote}{\arabic{footnote}}
and for $U$ is\footnote{The $\log n_1$ term comes from the number of iterations needed when using the power method to compute the largest (or smallest) eigenvector of a given matrix. 
} $O(n_1^2K\log n_1)$.
 Applying the $(1+2/e+\epsilon)$ approximate $\texttt{k-medians}$ has a complexity $O(f(K,\epsilon)n_1^{O(1)})$ where $f(K,\epsilon)=(\epsilon^{-2}K\log K)^K$, see \cite{approx_kmedians}. Here we used (approximate) $\texttt{k-medians}$ because it can be linked easily with $\ell_{2\to \infty}$ perturbation bounds (see Lemma 5.1 in \cite{lihua20}). But we could also apply (approximate) $\texttt{k-means}$ as a rounding step and use results from \cite{LuSC2020}, Section 2.4 for the analysis. Depending on the rounding step used, the dependencies in some model parameters such as the number of communities $K$ can change. 
\section{Main results}
 First, we derive a new concentration bound for the similarity matrix $B$. It improves the upper-bound $\sqrt{n_1n_2p_{max}^2}\vee \log n_1 $ used in \cite{braun2022minimax} to $\sqrt{n_1n_2p_{max}^2}$ when $n_1n_2p_{max}^2\gtrsim \log n_1$. This improvement of a $\sqrt{\log n_1}$ factor is essential to show that \spec\, achieves exact recovery in the challenging parameter regime where $n_1n_2p_{max}^2$ is of order $\log n_1$.
 \begin{theorem}\label{thm:improved_conc} Let $B=\calH(AA^\top)$ where $A\sim BiSBM(n_1,n_2,K,L,\Pi)$ with $n_1n_2p_{max}^2\gtrsim \log n_1$ and $n_2\gtrsim n_1\log^2 n_1$. Then with probability at least $1-n_1^{-\Theta(1)}$ \[ \norm{B-\expec(B)}\lesssim \sqrt{n_1n_2p_{max}^2}.\]
\end{theorem}
\begin{remark}By using this concentration inequality, one could improve the conditions of applicability of Proposition 1. and Theorem 2. For example, Proposition 1 requires that $n_1n_2p_{max}^2\geq C\log n_1$ for a constant $C>0$ large enough. But by using the concentration inequality of Theorem \ref{thm:improved_conc}, we would only require $n_1n_2p_{max}^2\geq c\log n_1$ for an arbitrary constant $c>0$. See also Remark 8 in \cite{braun2022minimax}.
\end{remark}

Finally, we show that \spec\, achieves exact recovery by proving the following $\ell_{2\to \infty}$ concentration bound for the top$-r$ eigenspace $U$ of $B$. Let us denote the $\ell_{2\to \infty}$ between two matrices of eigenvectors $U$ and $U^*\in \R^{n\times K}$ by\[ d_{2\to \infty}(U,U^*)= \inf_{O\in \R^{n_1\times r}, O^\top O=I}\norm{UO- U^*}.\]
%
\begin{theorem}\label{thm:main} Assume that $A\sim BiSBM(n_1,n_2,K,L,\Pi)$ with $n_2\gtrsim n_1\log^2 n_1$, $n_1n_2p_{max}^2\geq C \log n_1$ for a large enough constant $C>0$, and $n_2p_{max}^2=o(1)$. Let $U\Lambda U^\top$ (resp. $U^*\Lambda^*U^{*\top}$) be the spectral decomposition of $B=\calH(AA^\top)$ (resp. $B^*=PP^\top$). Then there exists a constant $c>0$ (that can be made arbitrarily small if $C$ is chosen large enough) such that with probability at least $1-n^{-\Theta(1)}$ \[ d_{2\to \infty}(U, U^*) \leq \frac{c}{\sqrt{n_1}}.\]
\end{theorem}
\begin{corollary}\label{cor:1} Under the same assumption as in Theorem \ref{thm:main} \spec\, achieves exact recovery with probability at least $1-n^{-\Theta(1)}$.  
\end{corollary}

\begin{corollary}\label{cor:2} Under the same assumptions of Theorem \ref{thm:main} with the choice $T=n_1n_2p_{max}^2/\log \log n_1$, \specc\, achieves exact recovery with probability at least $1-n^{-\Theta(1)}$.  
\end{corollary}

\section{Proof of Theorem \ref{thm:improved_conc}}
%
%
The proof strategy is based on the combinatorial argument developed by \cite{freige05}.

Let us denote  \[\calE =\left\lbrace  \max_{l\in [n_2]} \sum_{i\in [n_1]} A_{il}\leq C \sqrt{\log n_1}\right\rbrace .\] By Chernoff bound and a union bound \[ \prob(\calE^c)\leq n_2e^{-0.5C^2\frac{\log n_1}{n_1p_{max}}} \leq e^{\log n_2-0.5C^2\log(n_1)\sqrt{\frac{n_2}{n_1\log n_1}}}\leq e^{-\Omega(\log (n_1)}.\] By choosing $C$ large enough, we can ensure that $\calE$ occurs with probability at least $1-n_1^{-3}$. From now, we will condition on this event. 

\paragraph{Step 1.}
A standard $\epsilon-$net argument with the Euclidean norm (see e.g. Lemma B.1 and B.2 in \cite{RobHyp2020}) shows that for all $0<\epsilon<1/2$ there exists a $\epsilon-$net $\calN$ of $\mathbb{S}^ {n-1}$ such that $|\calN|\leq (1+\frac{2}{\epsilon})^n$ and \[ \norm{B-\expec(B)} \leq \frac{1}{1-2\epsilon} \sup_{x\in \calN} \left|x^\top (B-\expec(B))x\right|. \] In the following, we will fix $\epsilon = 1/4$.

\paragraph{Step 2.} In order to bound the previous quantity, let us introduce for all $x\in \mathbb{S}^{n_1-1}$ the set of ``light pairs'' \[ \calL(x) =\lbrace (i,j)\in [n_1]\times [n_1] : |x_ix_j|\leq \sqrt{\frac{n_2}{n_1}}p_{max} \rbrace\] and the set of ``heavy pairs" \[ \calH (x)= [n_1]\times [n_1] \setminus \calL(x).\] When clear from the context, we will omit the dependency in $x$ in the notations of the previous sets.

We have \[ \sup_{x\in \calN} \left|x^\top (B-\expec(B))x\right| \leq  \sup_{x\in \calN}\underbrace{\left|\sum_{(i,j)\in \calL}x_ix_jB_{ij}-x^\top\expec Bx\right|}_{(T1)} +\underbrace{\sup_{x\in \calN}\left|\sum_{(i,j)\in \calH}x_ix_jB_{ij}\right|}_{(T2)}.\]

\paragraph{Step 3.} We are going to bound (T1) w.h.p. Observe that \[ (T1) \leq \underbrace{\left|\sum_{(i,j)\in \calL}x_ix_j(B_{ij}-\expec B_{ij})\right|}_{(E1)}+\underbrace{\left|\sum_{(i,j)\in \calH}x_ix_j\expec B_{ij}\right|}_{(E2)}.\] It is easy to bound the deterministic quantity (E2)\begin{align*}
    (E2) &\leq \sum_{(i,j)\in \calH}\expec B_{ij}\frac{(x_ix_j)^2}{|x_ix_j|}\\
    &\leq \sqrt{\frac{n_1}{n_2}}p_{max}^{-1}n_2p_{max}^2\sum_{(i,j)\in \calL}(x_ix_j)^2\\
    &\leq \sqrt{n_1n_2}p_{max}\sum_{i\in [n_1]}x_i^2\sum_{j\in [n_1]}x_j^2\\
    &=\sqrt{n_1n_2}p_{max}.
\end{align*}

\rev{The upper-bound of (E1) conditioned on $\calE$ follows from Lemma \ref{lem:conc_prod_bin_gen} that gives \[ (E1) \leq C_1  \sqrt{n_1n_2}p_{max} \] with probability at least $1-e^{-11n_1}$ for some constant $C_1>1$. Since $|\calN|\leq e^{9n_1}$ according to Step 1, we obtain by a union bound argument that \begin{align*}
    \prob &\left(\calE \cap \sup_{x\in \calN} \left|\sum_{(i,j)\in \calL}x_ix_jB_{ij}-x^\top\expec Bx\right|>2C_1\sqrt{n_1n_2}p_{max} \right)  \\ 
    &\leq \prob\left(\calE \cap \sup_{x\in \calN}\left|\sum_{(i,j)\in \calL}x_ix_j(B_{ij}-\expec B_{ij})\right|>C_1\sqrt{n_1n_2}p_{max}\right) \\
    &\leq |\calN|e^{-11n_1} \leq e^{-2n_1}.
\end{align*} 
}

\paragraph{Step 3.} We are now going to bound the term involving the heavy pairs $(T2)$. First, one needs to control the sum of the entries of each row and column of $B$.
\begin{lemma}\label{lem:bounded_degree} There exists a constant $C_2>0$ such that with probability at least $1-e^{\Theta(n_1n_2p_{max}^2)}$ \[ \max_{i\in [n_1]}\sum_{j\in [n_1]} B_{ij} \leq C_2n_2n_1p_{max}^2. \]
\end{lemma}
\begin{proof}
Fix $i\in [n_1]$. We have $S=\sum_j B_{ij}=\langle A_{i:}, \sum_{j\neq i} A_{j:} \rangle$. One can apply Lemma \ref{lem:conc_prod_bin_gen2} (see appendix) with sets $I=\lbrace i \rbrace$ and $J=[n_1]\setminus I$. We conclude by using a union bound. 
\end{proof}

Then, we need to show that the matrix $B$ satisfies w.h.p. the discrepancy property defined below, with appropriate parameters.
 \begin{definition}
 Let $M$ be a $n\times n$ matrix with non-negative entries. For every $S,  T \subset [n]$, let $e_M(S,T)$ denote the number of edges between $S$ and $T$ \[ e_M(S,T)=\sum_{i \in S}\sum_{j\in T}M_{ij}.\] We say that $M$ obeys the discrepancy property DP$(\delta, \kappa_1, \kappa_2)$ with parameters
$\delta> 0$, $\kappa_1>0$ and $\kappa_2\geq 0$ if for all non-empty $S, T\subset [n]$, at least one of the following properties hold
\begin{enumerate}
    \item $e_M(S,T)\leq \kappa_1\delta |S||T|$;
    \item $e_M(S,T)\log \frac{e_M(S,T)}{\delta |S||T|}\leq \kappa_2 |S|\vee |T| \log \frac{en_1}{|S|\vee |T|}$.
\end{enumerate}
 \end{definition}
If one can show that $B$ satisfies $DP(\delta, \kappa_1, \kappa_2)$ where  $\kappa_1$, $\kappa_2>0$ are absolute constants  and $\delta = n_2p_{max}^2$, then Lemma B.4 in \cite{RobHyp2020} would imply that \[ (T2) \lesssim_{\kappa_1, \kappa_2} \sqrt{n_1n_2}p_{max}.\]
 Consequently, to bound (T2), it is sufficient to show that $B$ satisfies $DP(\delta, \kappa_1, \kappa_2)$. W.l.o.g., one can assume that $|S|\leq |T|$. If $|T|\geq \frac{n_1}{e}$ then Lemma \ref{lem:bounded_degree} leads to \[ \frac{e_B(S,T)}{\delta |S||T|}\leq \frac{|S|C_2n_2n_1p_{max}^2}{n_2p_{max}^2|S|n_1/e}\leq C_2e\] with probability at least $1-e^{\Theta(n_1n_2p_{max}^2)}$.
Otherwise, we can write $e_B(S,T) =\sum_{i,j}w_{ij}\langle A_{i:},A_{j:} \rangle$ where $w_{ii}=0$ and $w_{ij}=\indic_{i\in S}\indic_{j\in T}$ . By Lemma \ref{lem:conc_prod_bin_gen2} we have  \[ \prob \left(e_B(S,T) \geq C\delta |S||T| \right)\leq e^{-\frac{C}{2}\delta |S||T| }\] for all $C>C^*$ where $C^*>0$ is a large enough constant.
We can now continue as in the proof of Theorem 5.2 in \cite{lei2015}. For a given $c>0$, let us denote \[k(T,S)= \max\left( c\frac{|T| \log \frac{en_1}{ |T|}}{\delta |S||T|}, C^*\right).\] We have
\[ \prob \left(e_B(S,T) \geq k(T,S)\delta |S||T| \right)\leq e^{-\frac{c}{2}|T| \log \frac{en_1}{ |T|}}
\]
and thus
\begin{align*}
    &\prob\left(\exists S, T \subset [n_1], |S|\leq |T|: e_B(S,T) \geq k(T,S)\delta |S||T| \right) \\
    &\leq \sum_{I, J: |I|\leq |J|<n_1/e}e^{-\frac{c}{2}|T| \log \frac{en_1}{ |T|}}\\
    &\leq \sum_{s\leq t \leq n_1/e}\sum_{|S|=s,|T|=t}e^{-\frac{c}{2}|T| \log \frac{en_1}{ |T|}}\\
    &\leq \sum_{s\leq t \leq n_1/e}\binom{n_1}{s}\binom{n_1}{t}e^{-\frac{c}{2}t \log \frac{en_1}{ t}}\\
    &\leq \sum_{s\leq t \leq n_1/e}\left(\frac{en_1}{s}\right)^s\left(\frac{en_1}{t}\right)^te^{-\frac{c}{2}t \log \frac{en_1}{ t}}\\
    &\leq  \sum_{s\leq t \leq n_1/e} e^{-\frac{c}{2}t \log \frac{en_1}{ t}+t+s+t\log\frac{n_1}{t}+s\log\frac{n_1}{s}}\\
    &\leq \sum_{s\leq t \leq n_1/e}e^{-\frac{c}{2}t \log \frac{en_1}{ t}+2t+2t\log\frac{n_1}{t}}\\
    &\leq \sum_{s\leq t \leq n_1/e}e^{-\frac{c-8}{2}t \log \frac{en_1}{ t}}\\
    &\leq \sum_{s\leq t \leq n_1/e} n_1^{-\frac{c-8}{2}}\\
    &\leq n_1^{-\frac{c-12}{2}}
\end{align*}
by using repeatedly the fact that $t\log \frac{n_1}{t}$ is increasing on $[1,\frac{n_1}{e}]$. By choosing a constant $c>12$, we have shown that $B$ satisfies $DP(n_2p_{max}^2, \kappa_1,\kappa_2)$.

\paragraph{Step 4.} We can conclude by summing all the terms that have been shown to be $O(\sqrt{n_1n_2}p_{max})$ w.h.p.

\section{Entrywise analysis of the spectral method}

To show that the spectral method achieves exact recovery, we need to derive $\ell_{2\to \infty}$ eigenspace perturbation bound. Unfortunately, existing results only apply to symmetric matrices with independent entries or weak dependencies (see Section 7 in \cite{lihua20}) and cannot be directly applied to our setting. We propose an extension of the main result of \cite{lihua20} to the hollowed Gram matrix $B$ considered in this work. We believe that our result can be extended to more general Gram matrices or Kernel matrices.

\subsection{Notations and preliminary results}\label{subsec:intro}
First, let us introduce some notation. Let $\tilde{B}^*=\calH(PP^\top)$ and $B^*=PP^\top$. Let $\lambda_1\geq \ldots \geq \lambda_r$ (resp. $\lambda_1^*\geq \ldots \geq \lambda_r^*$) be the top$-r$ eigenvalues of $B$ (resp. $B^*)$ and $U$ (resp. $U^*$) the corresponding matrix of eigenvectors. 

The spectral decomposition of the matrices $B$ and $B^*$ is given by \[ B= \bar{U}\Lambda \bar{U}^\top, \quad  B^*= U^*\Lambda^* U^{* \top}\] where $\bar{U}$ is the full eigenspace matrix of $B$ and $\Lambda$ is the diagonal matrix of non-zero eigenvalues of $B$ (resp. $\Lambda^* =\diag (\lambda_1^*, \ldots, \lambda_r^*)$). The noise $E= B-B^*$ can be further decomposed as \[ \underbrace{\calH((A-P)(A-P)^\top)}_{\tilde{E}}+ \underbrace{\calH(P(A-P)^\top +(A-P)P^\top)}_{E'}+\underbrace{\tilde{B}^*-B^*}_{E''}. \]

First, let us establish analogous results to the Conditions (A2) and (A3) in \cite{lihua20}. 
\begin{lemma}\label{lem:lem_fond} Under the assumption of Theorem \ref{thm:main}, there is an absolute constant $C_1>0$, such that for any $W\in \R^{n\times K}$, the following inequalities hold with probability at least $1-n_1^{-\Theta(1)}$.
\begin{enumerate}
    \item $\norm{\Lambda -\Lambda^*}\leq C_1 \sqrt{n_1n_2p_{max}^2}$,
    \item $\norm{E}_{2\to \infty} \leq C_1 \sqrt{n_1n_2p_{max}^2}$,
    \item $\norm{EU^*}\leq C_1 \sqrt{n_1n_2p_{max}^2}$,
    \item  $\max_i\norm{E_{i:}W}\leq b_{\infty}(\delta)\norm{W}_{2\to \infty}+b_2(\delta)\norm{W}$, where $b_\infty(\delta) = C_1\frac{R(\delta)}{ \log R(\delta)} $ and $b_2(\delta)= C_1\frac{\sqrt{n_2p_{max}^2}R(\delta)}{\log R(\delta)}$ with $R(\delta)= \log (n_1/\delta)+K$ and $\delta=n_1^{-c}$ for some constant $c>0$.
\end{enumerate}
\end{lemma}

\begin{proof} Recall that by Theorem \ref{thm:improved_conc}, we have with probability at least $1-n_1^{-\Theta(1)}$ \[ \norm{\tilde{E}}\lesssim \sqrt{n_1n_2p_{max}^2}. \] Also, by definition, $\max_i \norm{P_{i:}}^2\leq n_2p_{max}^2 $ so $\norm{E''}\leq \sqrt{n_2}p_{max}=o(\sqrt{n_1n_2p_{max}^2})$. By Lemma 5 in \cite{braun2022minimax} we have $\norm{(A-P)Z_2}\lesssim \sqrt{n_1n_2p_{max}}$, hence by submultiplicativity of the norm \[ \norm{E'}\lesssim \sqrt{n_1n_2p_{max}} \norm{\Pi Z_1^\top}\lesssim \sqrt{n_2p_{max}}n_1p_{max} = O(\sqrt{n_1n_2p_{max}^2})\] because $n_1p_{max}=O(1)$ by assumption. Consequently, the dominant error term is $\tilde{E}$ and we have shown that \begin{equation}\label{eq:norm_error}
    \norm{E}\lesssim \sqrt{n_1n_2p_{max}^2}.
\end{equation}
\paragraph{Proof of 1.} This is a direct consequence of Weyl's inequality and \eqref{eq:norm_error}.

\paragraph{Proof of 2.} It follows from the fact that $\norm{E}_{2\to \infty} \leq \norm{E}$.

\paragraph{Proof of 3.} It is a direct consequence of the sub-multiplicativity of the norm and the fact that $\norm{U^*}\leq 1$.

\paragraph{Proof of 4.} By Proposition 2.2 in \cite{lihua20}, if we can show that for any $\delta \in (0,1)$ and vector $w\in \R^{n_1}$ there exist $a_\infty(\delta), a_2(\delta)>0$ such that for each $i\in [n_1]$ 
\[ E_{i:}w\leq a_\infty(\delta) \norm{w}_\infty +a_2(\delta) \norm{w} \]
with probability at least $1-\delta$, then we can choose $b_\infty(\delta)=2a_\infty(\frac{\delta}{5^Kn_1})$ and $b_2(\delta)=2a_2(\frac{\delta}{5^Kn_1})$. Fix $w\in \R^{n_1}, i\in [n_1]$ and let us denote $R=A-P$. Consider $S = \tilde{E}_{i:}w =\sum_{j\in [n_2]\setminus{\lbrace i\rbrace }}\langle R_i,R_j \rangle w_j$. Conditionally on $R_i$, this is a sum of independent and centered r.v.s. By using Lemma F.3 in \cite{lihua20} with weights $\tilde{w}_{jl}=R_{il}w_j$ we obtain that conditionally on $R_i$ the following holds with probability at least $1-\delta$
\[ S\leq f(\delta)\left( \norm{\tilde{w}}_\infty+\sqrt{\sum_{j,k}\tilde{w}_{jl}^2p_{jl}}\right)\]
where $f(\delta)=\frac{2\log (1/\delta)}{F^{-1}(2\log(1/\delta)}$ and $F(t)=t^2e^t$. But $\norm{\tilde{w}}_\infty \leq \norm{w}_\infty$ and 
\[ \norm{\tilde{w}}^2=\sum_{j,l}R_{il}^2w_j^2=\norm{w}^2\norm{R_{i:}}^2. \] 
Besides, with probability at least $1-e^{\Theta(n_2p_{max})}$, $\norm{R_{i:}}^2 \leq Cn_2p_{max} $ by Hoeffding's inequality. Therefore with probability at least $1-\delta - e^{-\Theta(n_2p_{max})}$ \[ S\leq f(\delta)\left( (\norm{w}_\infty+\sqrt{n_2p_{max}^2}\norm{w}\right). \] 
It remains to bound $S'= E'_{i:}w =\sum_{j\in [n_2]\setminus{\lbrace i\rbrace }}(\langle R_i,P_j \rangle+ \langle P_i,R_j \rangle) w_j $ and $S''= E''_{i:}w $. We have
\[S''= \norm{P_{i:}}^2w_i\leq \norm{w}_\infty n_2p_{max}^2=o(\norm{w}_\infty).\]
Also observe that $=\sum_{j\in [n_2]\setminus{\lbrace i\rbrace }}\langle R_i,P_j \rangle w_j=\sum_{j\neq i,l}R_{jl}w_jP_{il}$, so we can apply Lemma F.3 in \cite{lihua20} with weights $(w_jP_{il})_{j\neq i,l}$. We obtain
\[ S''\leq f(\delta)p_{max}\left( \norm{w}_\infty+ \norm{w}\sqrt{n_2p_{max}}\right)\]
with probability at least $1-\delta$. A similar result holds for $S'=\sum_{j\neq i,l} \langle P_i,R_j \rangle w_j$: we can apply again Lemma F.3 in \cite{lihua20} with weights $(P_{jl}w_j)_{j,l}$ and obtain
\[ S' \leq f(\delta) p_{max}\left(  \norm{w}_\infty+ \norm{w}\sqrt{n_2p_{max}}\right).\]

So we can choose $a_\infty(\delta)= f(\delta)$ and $a_2(\delta)=f(\delta)\sqrt{n_2p_{max}^2}$. One can check that, as in Lemma 3.1 in \cite{lihua20},  $b_\infty(\delta) = \frac{4R(\delta)}{\log R(\delta)}$ and $b_2(\delta)= 4\frac{\sqrt{n_2p_{max}^2}R(\delta)}{\log R(\delta)}$. Also note that $n_2p_{max}\gtrsim \log n_1$ by assumption so if we choose $\delta = n_1^{-c}$ for an appropriate constant $c>0$, the term $e^{-\Theta(n_2p_{max})}$ will be negligible compared to $\delta$.
\end{proof}

\subsection{A new decoupling argument}\label{subsec:decoupling}
The main difficulty to adapting Theorem 2.3 and 2.5 \cite{lihua20} comes from the decoupling assumption (A1) which requires the existence of a matrix $B^{(i)}$ (typically obtained by replacing the $i$-th row and column of $B$ by zeros or the expectation of the entries) such that for any $\delta\in (0,1)$ \begin{equation}\label{eq:a1}
    d_{TV}(\prob_{(B_i,B^{(i)})}, \prob_{B_i}\times \prob_{B^{(i)}})\leq \frac{\delta}{n}.
\end{equation}
If the matrix $B$ had independent entries it would be straightforward to satisfy this condition, but in our setting, it is not clear how to obtain such a general result. Consequently, we adopted a different approach that avoids bounding the total variation distance between two probability distributions.  

Let us denote by $B^{(i)}$ the matrix obtained  by removing the $i$-th row and column of $B$. We have \[ \norm{B^{(i)}-B} \leq \norm{B_{i:}} \leq \norm{E_{i:}} + \norm{B^*_{i:}} \lesssim \sqrt{n_1n_2p_{max}^2} \] with probability at least $1-e^{\Theta(n_1n_2p_{max}^2)}$, since $\norm{B^*_{i:}}\leq \sqrt{n_1}n_2p_{max}^2$, $n_2p_{max}^2=o(1)$, and $\norm{E_{i:}}\leq \norm{E} \leq \sqrt{n_1n_2p_{max}^2}$.

We also have by definition \begin{align*}
    \norm{(B^{(i)}-B)U)}&\leq \norm{B_{i:}U}+\norm{B_{i:}U_{i:}^\top}\\
    &\leq \norm{(BU)_{i:}}+\norm{B_{i:}}\norm{U_{i:}}\\
    &\leq \norm{U_{i:}\Lambda}+\norm{B_{i:}}\norm{U_{i:}}\\
    &\leq (\norm{\Lambda}+\norm{B_{i:}})\norm{U_{i:}}.
\end{align*}
 Hence, because of assumption \ref{ass:eig_low_bd} we obtain that w.h.p. \[ \frac{ \norm{(B^{(i)}-B)U)}}{\lambda_r^*}\lesssim (1+\frac{1}{\sqrt{n_1n_2p_{max}^2}}) \norm{U_{i:}} \lesssim \norm{U}_{2\to \infty}.\]
 
 These inequalities correspond to the Condition (C0) used in the proof of Theorem 2.3 in \cite{lihua20}. They are summarized in the following lemma.
 \begin{lemma} The following inequalities hold with probability at least  $1-e^{\Theta(n_1n_2p_{max}^2)}$
 \begin{enumerate}
     \item  $\norm{B^{(i)}-B} \lesssim \sqrt{n_1n_2p_{max}^2}$,
     \item $\frac{ \norm{(B^{(i)}-B)U)}}{\lambda_r^*}\lesssim \norm{U_{i:}}\lesssim \norm{U}_{2\to \infty}.$
 \end{enumerate}
 
 \end{lemma}
 
 Steps one and two of the proof of Theorem 2.3 \cite{lihua20} are deterministic and still hold in our setting (see the discussion in Section \ref{subsec:concl}). The only step that uses the decoupling argument is the third step where one needs to bound $\norm{E_{i:}(U^{(i)}H^{(i)}-U^*)}$ where $H^{(i)}\in \R^{r\times r}$ is the orthogonal matrix that best aligns $U^{(i)}$ and $U^*$. 
 
 
 \begin{lemma}\label{lem:4} Let $W^{(i)}\in \R^{n_1\times K}$ be a matrix that only depends on $B^{(i)}$. Under the assumptions on Theorem \ref{thm:main}, it holds with probability at least $1-n_1^{-c'}$ for  some constant $c'>0$ that for all $i\in [n_1]$ \[ \norm{E_{i:}W^{(i)}}\lesssim \frac{\log n_1}{\log \log n_1} \norm{W^{(i)}}_{2\to \infty}+\frac{\sqrt{n_2}p_{max}\log n_1}{\log \log n_1} \norm{W^{(i)}}.\] 
 \end{lemma}
 
\begin{proof}
Recall that $E=\tilde{E}+E'+E''$. By triangular inequality \[ \norm{E_{i:}W^{(i)}} \leq \norm{\tilde{E}_{i:}W^{(i)}}+\norm{E'_{i:}W^{(i)}}+\norm{E_{i:}''W^{(i)}}.\]
We will first handle the first term.  Let us denote $R=A-P$ and consider \[S=\tilde{E}_{i:}w^{(i)} = \sum_{j\in [n_1]\setminus{i},l\in [n_2]}R_{il}R_{jl}w_j^{(i)}\]  where $w^{(i)}\in \R^{n_1}$ depends on $A_{-i}$. 

Conditionally on $A_{-i}$, $S$ is a weighted sum of independent and centered Bernoulli's r.v. Hence, by Lemma F.3 in \cite{lihua20} with $\delta=n_1^{-c}$, and weights $\tilde{w}_{jl}=R_{jl}w_j^{(i)}$ we obtain 
\[ \prob\left( S \gtrsim \frac{\log n_1}{\log \log n_1}\left(\norm{\tilde{w}}_\infty+\sqrt{p_{max}}\norm{\tilde{w}}\right)\middle| A_{-i}\right)\leq n^{-c}.\]
Since $R_{jl}\leq 1$, we have $\norm{\tilde{w}}_\infty\leq \norm{w}_\infty$. By definition we have \[ \norm{\tilde{w}}^2=\sum_{j\neq i,l}R_{jl}^2(w_j^{(i)})^2.\]

\paragraph{Fact.}We have with probability at least $1-e^{-\Theta(n_2p_{max})}$ that \[ \max_{j\neq k}\sum_l R_{jl}^2 \lesssim n_2p_{max}.\]

\begin{proof}[Proof of the Fact.]
We have $R_{jl}^2\leq 1$ and $\Var(\sum_l R_{jl}^2)\leq 2n_2p_{max}$. Hence by Bernstein inequality, \[ \prob(|\sum_l R_{jl}^2 -\expec(\sum_l R_{jl}^2)|\gtrsim n_2p_{max} )\leq e^{-\Theta(n_2p_{max})}.\] We can conclude by a union bound and the fact that $n_2p_{max}\gtrsim \log n_1$ by assumptions on the sparsity level $p_{max}$ and $n_2\gtrsim n_1\log n _1$.
\end{proof}

Let us denote by $\Omega_1$ the event under which the inequality of the previous fact holds. Note that this event only depends on $A_{-i}$. We have $\prob(\Omega_1^c)\leq e^{-\Theta(n_2p_{max})}$. Consequently \begin{align*}
     \prob&\left(  S \gtrsim \frac{\log n_1}{\log \log n_1}\left(\norm{w^{(i)}}_\infty+\sqrt{n_2}p_{max}\norm{w^{(i)}}\right)\right)\\
     &\leq \prob\left( S \gtrsim \frac{\log n_1}{\log \log n_1}\left(\norm{w^{(i)}}_\infty+\sqrt{n_2}p_{max}\norm{w^{(i)}}\right) \middle|\Omega_1 \right) + e^{-\Theta(n_2p_{max})}\\
     &\leq \expec_{\Omega_1}\prob\left( S \gtrsim \frac{\log n_1}{\log \log n_1}\left(\norm{w^{(i)}}_\infty+\sqrt{n_2}p_{max}\norm{w^{(i)}}\right) \middle|A_{-i} \right)+e^{-\Theta(n_2p_{max})}\\
     &\leq \expec_{\Omega_1}\prob\left( S \gtrsim \frac{\log n_1}{\log \log n_1}\left(\norm{\tilde{w}}_\infty+\sqrt{p_{max}}\norm{\tilde{w}}\right) \middle|A_{-i} \right)+e^{-\Theta(n_2p_{max})}\\
     &\leq n_1^{-c}+e^{-\Theta(n_2p_{max})}
\end{align*}
where $\expec_{\Omega_1}$ denotes the expectation over $A_{-i}$ conditioned on $\Omega_1$.

The other terms $E'_{i:}w^{(i)}$ and $E_{i:}''w^{(i)}$ can be handled in a similar way. They are actually easier to treat because one doesn't need to use a conditioning argument since $E'_{i:}$, $E''_{i:}$ are independent of $A_{-i}$. 

We have by definition \[E''_{i:}w^{(i)}\leq n_2p_{max}^2\norm{w^{(i)}}_\infty \ll \frac{\log n_1}{\log \log n_1}\norm{w^{(i)}}_\infty.\]
Also we can decompose \[ E'_{i:}w^{(i)}= \underbrace{\sum_{j\neq i,l}R_{il}P_{jl}w^{(i)}_j}_{S_1}+\underbrace{\sum_{j\neq i,l}P_{il}R_{jl}w^{(i)}_j}_{S_2}.\]
$S_1$ is a sum of $n_2$ weighted independent Bernoulli's r.v.  with weights given by $w_l=\sum_{j\neq i}P_{jl}w_j^{(i)}$. Lemma F.3 in \cite{lihua20} gives with probability at least $1-n_1^{-c}$ \begin{align*}
    S_1 &\lesssim \frac{\log n_1}{\log \log n_1}\left(\norm{w}_\infty+\sqrt{p_{max}}\norm{w} \right)\\
&\lesssim \frac{\log n_1}{\log \log n_1}\left(n_1p_{max}\norm{w^{(i)}}_\infty+ \sqrt{n_1}p_{max} \norm{w^{(i)}}\right).
\end{align*} 

By a similar argument, we can show that with probability at least $1-n_1^{-c}$  \[ S_2 \lesssim  \frac{\log n_1}{\log \log n_1} \left( p_{max}\norm{w^{(i)}}_\infty+\sqrt{n_2}p_{max}^{1.5}\norm{w^{(i)}}\right).\] 

Consequently, with probability at least $1-O(n_1^{-c})$, \[ E_{i:}w^{(i)} \lesssim \frac{\log n_1}{\log \log n_1}\left(\norm{w^{(i)}}_\infty+\sqrt{n_2}p_{max}\norm{w^{(i)}}\right).\]

Then, by using Proposition 2.2 ($\epsilon$-net argument)  in \cite{lihua20} we obtain that with probability at least $1-O(n_1^{-c})$  \begin{equation}\label{eq:step3}
    \norm{E_{i:}W^{(i)}}\lesssim  \frac{\log n_1}{\log \log n_1}\norm{W^{(i)}}_{2\to \infty}+ \frac{\sqrt{n_2}p_{max}\log n_1}{\log \log n_1} \norm{W^{(i)}}.
\end{equation} 
Once we have obtained this inequality, the proof of Step III. is the same as in \cite{lihua20}.
\end{proof}

\subsection{Proof of Theorem \ref{thm:main}} \label{subsec:concl}

First, we will extend Theorem 2.3 in \cite{lihua20}. In order to make the adaptation easier, we will use the same notations as in \cite{lihua20}.
Let $\Delta^*=\lambda_{min}^*$ be the effective eigengap (it corresponds with the definition in \cite{lihua20}, with $s=0$). In our setting, the condition number $\bar{\kappa}$ only depends on $K$ and $L$ and hence is considered a constant.  Also, observe that $U^*$ is the full eigenspace of $B^*$. We have shown in Section \ref{subsec:intro} and \ref{subsec:decoupling} that the following conditions (partially matching the assumptions (A1)-(A4) in \cite{lihua20}) hold with $\delta=n_1^{-q}$ for some constant $q>0$.

\begin{condition}\label{cond:c1} There exists a constant $C_1>0$ such that with probability at least $1-O(n_1^{-q})$ the following conditions hold \begin{enumerate}
    \item $\norm{B^{(i)}-B}\leq L_1(\delta):= C_1\sqrt{n_1n_2p_{max}^2}$,
    \item  $\frac{ \norm{(B^{(i)}-B)U)}}{\lambda_r^*}\leq C_1 \norm{U}_{2\to \infty}$.
\end{enumerate}
 With the notations of \cite{lihua20}, the functions $L_2(\delta),L_3(\delta)$ and $\kappa(\Lambda^*)=\bar{\kappa}$ that appears in Assumption (A1) are $O(1)$.
\end{condition}

\begin{condition}\label{cond:c2}
    There exists a constant $C_2>0$ such that with probability at least $1-O(n_1^{-q})$ the following inequalities hold
    \begin{enumerate}
            \item $\norm{\Lambda -\Lambda^*}\leq \lambda_-(\delta):=C_2 \sqrt{n_1n_2p_{max}^2}$,
            \item $\norm{EU^*}\leq E_+(\delta):=C_2 \sqrt{n_1n_2p_{max}^2}$,
            \item $\norm{E}_{2\to \infty} \leq E_\infty(\delta)=C_2 \sqrt{n_1n_2p_{max}^2}$.
    \end{enumerate}
\end{condition}

\begin{condition}\label{cond:c3}
    For any $i\in [n_1]$ and fixed matrix $W\in \R^{n_1\times r}$, \[ \norm{E_{i:}W}\leq b_{\infty}(\delta)\norm{W}_{2\to \infty}+b_2(\delta)\norm{W}, \text{ with probability at least } 1-O(n_1^{-q})\] where $b_\infty(\delta)\lesssim \frac{\log n_1}{\log \log n_1}$ and $b_2(\delta) \lesssim \frac{\sqrt{n_2}p_{max}\log n_1}{\log \log n_1}$.
\end{condition}

\begin{condition}\label{cond:c4}
    We have $\Delta^*\geq 4(\sigma(\delta)+L_1(\delta)+\lambda_-(\delta))$ where $\sigma(\delta)=E_\infty(\delta) + b_\infty(\delta)+b_2(\delta)+E_+(\delta)$.
\end{condition}

\begin{theorem}\label{thm:lihua_ext} Let $\delta=n_1^{-q}$ for some constant $q>0$. Then under conditions C1-C4 and the assumptions of Theorem \ref{thm:main}, there exists a constant $C_3>0$ such that with probability at least $1-O(n_1^{-q})$ \begin{align*}
    d_{2\to \infty}(U,BU^*(\Lambda^*)^{-1}) \leq &\frac{C_3}{\Delta^*}\sigma(\delta)\left( \norm{U^*}_{2\to \infty}+\frac{\norm{EU^*}_{2\to \infty}}{\lambda_{min}^*}\right)\\
    &+ \frac{C_3}{\Delta^*} \left(\frac{E_+(\delta)b_2(\delta)}{\lambda_{min}^*}+ \frac{E_+(\delta)}{\sqrt{n_1}}\right) .
\end{align*} 
\end{theorem}
\begin{proof}
We cannot directly apply Theorem 2.3 in \cite{lihua20} because Condition C1 doesn't include the condition stated in \eqref{eq:a1}. But this condition is only used in the Step III. of Theorem 2.3 where one needs to control $\norm{E_{i:}(U^{(i)}H^{(i)}-U^*)}$. We used a different argument to control this quantity in Section \ref{subsec:decoupling} and we obtained by equation \eqref{eq:step3} \[ \norm{E_{i:}(U^{(i)}H^{(i)}-U^*)}\leq  b_\infty(\delta)\norm{(U^{(i)}H^{(i)}-U^*)}_{2\to \infty}+ b_2(\delta) \norm{(U^{(i)}H^{(i)}-U^*)}.\] This concludes Step III in Theorem 2.3 in \cite{lihua20}.
\end{proof}

\begin{corollary}  Under the same assumption as in Theorem \ref{thm:lihua_ext}, there is a constant $c>0$ (possibly depending on $q$) such that with probability at least $1-O(n_1^{-q})$ \[ d_{2\to \infty}(U,U^*) \leq \frac{c}{\sqrt{n_1}}.\]
\end{corollary}
\begin{proof}
By triangular inequality \[ d_{2\to \infty}(U,U^*)\leq d_{2\to \infty}(U,BU^*(\Lambda^*)^{-1})+d_{2\to \infty}(BU^*(\Lambda^*)^{-1},U^*). \] Notice that $U^*=B^*U^*(\Lambda^*)^{-1}$, so \[ d_{2\to \infty}(BU^*(\Lambda^*)^{-1},U^*) \leq \norm{(B-B^*)U^*(\Lambda^*)^{-1}}_{2\to \infty} \leq \frac{\norm{EU^*}_{2\to \infty}}{\lambda_{min}^*}.\]
%
We can bound $\norm{EU^*}_{2\to \infty}$ by using the same proof technique as in Lemma \ref{lem:lem_fond}, bullet 4, similarly to Lemma 3.3. in \cite{lihua20}. We obtain with probability at least $1-n_1^{-q}$ \[ \norm{EU^*}_{2\to \infty} \lesssim \log n_1 \norm{U^*}_{2\to \infty}+ \sqrt{p_{max}\log n_1}. \] Hence with probability at least $1-n_1^{-q}$
\[\frac{\norm{EU^*}_{2\to \infty}}{\lambda_{min}^*}\lesssim \frac{\log n_1}{n_1n_2p_{max}^2}\frac{1}{\sqrt{n_1}}.\]
It is easy to check that \begin{align*}
    &\sigma(\delta) = O(\frac{\log n_1}{\log \log n_1})\\
    &\frac{E_+(\delta)b_2(\delta)}{\Delta^*\lambda_{min}^*}=O(\frac{1}{\sqrt{n_1}\log \log n_1}) \\
    &\frac{E_+(\delta)}{\Delta^*\sqrt{n_1}}=O(\frac{1}{\sqrt{\log n_1}}\frac{1}{\sqrt{n_1}}).
\end{align*}

By consequence, triangular inequality and Theorem \ref{thm:lihua_ext} implies that w.h.p. \[d_{2\to \infty}(U,U^*) \leq \frac{c}{\sqrt{n_1}} \] for a constant $c>0$ that can be made small enough if the constant $C$ such that $n_1n_2p_{max}^2\geq C \log n_1$ is chosen large enough.

\end{proof}

\subsection{Proof of Corollary \ref{cor:1}}
The proof is standard, but for completeness, we outline it. First, we need to relate the $\texttt{k-medians}$ algorithm with the $\ell_{2\to \infty}$ perturbation bounds. It can be done by the following lemma.
\begin{lemma}[\cite{lihua20}]\label{lem:k-medians}
 Let $U, U^*\in \R^{n\times r}$ be two matrices with orthonormal columns. Then the $\texttt{k-medians}$ algorithm exactly recovers the clusters $\calC_1,\ldots, \calC_K$ if \[ d_{2\to \infty}(U,U^*)\leq \frac{1}{6\alpha}\min_{i,j\in [n_1]: z_i\neq z_j} \norm{U^*_{i:}-U^*_{j:}}.\]
\end{lemma}
Since by Theorem \ref{thm:main} we have $d_{2\to \infty}(U,U^*)\leq \frac{c}{\sqrt{n_1}}$ and by Assumption \ref{ass:eig_low_bd} $\min_{i,j\in [n_1]: z_i\neq z_j} \norm{U^*_{i:}-U^*_{j:}}\geq \frac{c_1}{\sqrt{n_1}}$, the assumption of Lemma \ref{lem:k-medians} holds whenever $c_1/6\alpha > c$.

\subsection{Proof of Corollary \ref{cor:2}}
It is sufficient to show that w.h.p. we have $\hat{r}=r$. But this is a straightforward consequence of Weyl's inequality and the fact that $\norm{B-B^*}\lesssim \sqrt{n_1n_2p_{max}^2}$.

\clearpage
\paragraph{Acknowledgement} The work leading to the preliminary version of this manuscript was done while G.B was a PhD student at Inria Lille in the MODAL team. G.B. would like to thank Hemant Tyagi for giving feedback on the preliminary version of the manuscript.  Special appreciation is also owed to Yizhe Zhu for diligently pinpointing inaccuracies within the proofs presented in the preceding version of the manuscript.
\bibliography{references}
\newpage
\appendix
\section{General concentration inequalities}
In this section, we provide proofs of the lemmas stated in the main text.

\rev{
\begin{lemma}\label{lem:conc_prod_bin_gen} Assume that the assumption of Theorem \ref{thm:improved_conc} are satisfied. Let us denote $S=\sum_{i,j\in [n_1]}w_{ij}\langle A_{i:},A_{j:} \rangle$ where $w_{ii}=0$ for all $i$, and $w_{ij}=x_ix_j\indic_{(i,j)\in \calL (x)}$ where $\norm{x}=1$ and $\calL(x)$ is the set of light pairs as defined in the proof of Theorem \ref{thm:improved_conc}.  In particular, $\norm{w}\leq 1 $, $\norm{w}_\infty \leq \sqrt{\frac{n_2}{n_1}}p_{max} $. Recall that $\calE =\left\lbrace  \max_{l\in [n_2]} \sum_{i\in [n_1]} A_{il}\leq C \sqrt{\log n_1}\right\rbrace $.
We have 
\[ \prob \left(\calE\cap \lbrace|S-\expec S| \gtrsim \sqrt{n_1n_2}p_{max}\rbrace \right)\leq e^{-11n_1}.\] 
\end{lemma}
}
\begin{proof}
We will use a similar decoupling approach as the one used in the proof of Hanson-Wright inequality, see \cite{HS}. Let $(\delta_i)_{i\in [n_1]}$ be independent Bernoulli's r.v. with parameter $1/2$ and let us define the set of indices
\[ \Lambda_\delta = \lbrace i\in [n_1]: \delta_i=1\rbrace\] and the random variable \[ S_\delta =\sum_{i,j}\delta_i(1-\delta_j)w_{ij}\langle A_{i:},A_{j:} \rangle= \sum_{i\in \Lambda_\delta} \langle A_{i:}, \sum_{j\in \Lambda_\delta^c}w_{ij}A_{j:}\rangle.\]
Note that $\expec _\delta S_\delta =S/4$.
 To simplify the notations we will denote by $\expec_{\Lambda^c}(.)$ (resp. $\expec_{\Lambda^c}(.)$) the expectation over $(A_{i:})_{i \in \Lambda^c_{\delta}}$ conditionally on $\delta$ and $(A_{i:})_{i \in \Lambda_{\delta}}$, $S_\delta$ (resp. the expectation over $(A_{i:})_{i \in \Lambda_{\delta}}$ conditionally on $\delta$ and $(A_{i:})_{i \in \Lambda^c_{\delta}}$, $S_\delta$). 

\paragraph*{Upper bound of the m.g.f. of $S_\delta$ conditionally on $\Lambda_\delta$.}

Conditionally on $\delta$ and $(A_{i:})_{i \in \Lambda_{\delta}}$, $S_\delta$ is a weighted sum of independent Bernoulli's r.v:
\[
S_\delta=\sum_{l\in [n_2]}\sum_{j\in \Lambda_\delta^c}A_{jl}\left(\sum_{i\in \Lambda_\delta}w_{ij}A_{il}\right).
\]
Hence, for all $t>0$ we have
\begin{align*}
    \log \expec_{\Lambda^c} &\left( e^{t(S_\delta-\expec_{\Lambda^c}(S_\delta))} \right) = \\
     &\log \expec_{\Lambda^c}\left( e^{tS_\delta}\right)-\sum_{i \in \Lambda_\delta}\sum_{j \in \Lambda_\delta^c}\sum_{l\in [n_2]}w_{ij}tA_{il}p_{jl}\\
    &=\sum_{j \in \Lambda_\delta^c}\sum_{l\in [n_2]} \left(\log (e^{t\sum_{i \in \Lambda_\delta}w_{ij}A_{il}}p_{jl}+1-p_{jl})-t\sum_{i \in \Lambda_\delta}w_{ij}A_{il}p_{jl}\right)\\
    &\leq \sum_{j \in \Lambda_\delta^c}\sum_{l\in [n_2]}\left(p_{jl}(e^{t\sum_{i \in \Lambda_\delta}w_{ij}A_{il}}-1)-t \sum_{i \in \Lambda_\delta}w_{ij}A_{il}p_{jl} \right)\tag{$\log(1+x)\geq x$, for all $x>-1$}\\
    &\leq  p_{max}t^2\sum_{j \in \Lambda_\delta^c}\sum_{l\in [n_2]}\frac{e^{t\norm{A_{:l}}_1\norm{w}_\infty}}{2}(\sum_{i \in \Lambda_\delta}A_{il}w_{ij})^2 \tag{by Taylor-Lagrange formula}\\
    .
\end{align*}
In order to upper-bound this m.g.f, it is necessary to control $\sum_{i\in [n_1]} A_{il}$ for each $l$. Since $\expec(\sum_{i\in [n_1]} A_{il})=O(n_1p_{max})=o(1)$ one can expect that these sums are generally of constant order. Unfortunately, this is not always the case and one needs to carefully control the number of indexes $l$ such that $\sum_{i\in [n_1]} A_{il}$ scales as $\sqrt{\log n_1}$. Toward this perspective, let us introduce the events  
 \[ L_1 =\left\lbrace l\in [n_2]: \sum_{i\in [n_1]} A_{il}\leq M\right\rbrace\] and 
 \[ L_2 =\left\lbrace l\in [n_2]: M< \sum_{i\in [n_1]} A_{il}\leq C\sqrt{\log n_1}\right\rbrace\] where $M>0$ is a constant that will be defined later. Note that onditionally on $\calE$, $[n_2]\subset L_1\cup L_2$. By using the fact that $w_{ij}=x_ix_j$ and $\sum_{j \in \Lambda_\delta^c}x_j^2\leq 1$ we obtain
 \[  \log \expec_{\Lambda^c} \left( e^{t(S_\delta-\expec_{\Lambda^c}(S_\delta))} \right) \leq  p_{max}t^2\frac{e^{tM\norm{w}_\infty}}{2}\sum_{l\in L_1}(\sum_{i \in \Lambda_\delta}A_{il}x_i)^2+ p_{max}t^2\frac{e^{tC\sqrt{\log n_1}\norm{w}_\infty}}{2}\sum_{l\in L_2}(\sum_{i \in \Lambda_\delta}A_{il}x_i)^2.\]

\paragraph{Control of the size of $L_2$.}
 Let \[ Y= \sum_l \indic_{\lbrace \sum_iA_{il} \geq M\rbrace}\] be the r.v. corresponding to the size of $L_1^c$. By Chernoff bound,
 \[ \prob\left( \sum_iA_{il} \geq M\right) \leq \prob\left( \sum_iA_{il} \geq M-4\right) \leq e^{-\frac{(M-4)^2}{2n_1p}} \leq e^{-C\sqrt{\log n_1}} \]
 (by choosing $M$ such that $(M-4)^2\geq 2C$) since $n_1p=\sqrt{\frac{n_1\log n_1}{n_2}}$ and $n_2\gtrsim n_1\log^2n_1$. Hence, by using Bernstein inequality, we obtain \[ \prob\left(Y -\expec(Y)\gtrsim n_2e^{-C\sqrt{\log n_1}}\right) \leq e^{-\Omega(n_2)}\] since $\sqrt{\log n_1}  \ll \log n_2 $. As a consequence the event
 \[ \calE' = \left\lbrace |L_1^c|\gtrsim n_2e^{-C\sqrt{\log n_1}}\right\rbrace \] 
 occurs with probability at most $e^{-\Omega(n_2)}$.

\paragraph{Control of the term $ \sum_{l\in [L_1]}(\sum_{i \in \Lambda_\delta}A_{il}x_{i})^2$.}
Let us define the event \[ \calE_1 = \left\lbrace \sum_{l\in L_1}(\sum_{i \in \Lambda_\delta}A_{il}x_{i})^2 \geq C_1 n_2p_{max}\right\rbrace\] for a constant $C_1>0$ large enough. Let us denote $Y_l=\sum_{i \in \Lambda_\delta}A_{il}x_{i}$. By assumption, $(Y_l^2\indic_{l \in L_1})_l$ are independent and $Y_l^2\indic_{l \in L_1}\leq M^2\sqrt{\frac{n_2}{n_1}}p_{max}$. Besides, we have $\expec(Y_l^2)\lesssim p_{max}$ and \begin{align}
    \expec(Y_l^4) &= \sum_{i_1, \ldots, i_4}\expec(A_{i_1l}A_{i_2l}A_{i_3l}A_{i_4l})x_{i_1}x_{i_2}x_{i_3}x_{i_4}\label{eq:a1}\\
    &\leq p\sum_i x_i^4 + 6p^2(\sum_ix_i^2)^2+4p^3\sum_ix_i\sum_ix_i^3+p^4(\sum_ix_i)^4\nonumber \\
    & \lesssim \sqrt{\frac{n_2}{n_1}}p_{max}^2+p_{max}^3\sqrt{n_1}+p_{max}^4n_1^2\nonumber\\
    &\lesssim \sqrt{\frac{n_2}{n_1}}p_{max}^2.\nonumber
\end{align}
By consequence, Berstein inequality implies that \begin{align*}
    \prob(\calE_1)&\leq \prob\left(\sum_{l\in [n_2]}\indic_{l\in L_1}(Y_l^2-\expec Y_l^2)\gtrsim n_2p_{max}\right)\\
    &\leq e^{-c\min\left(\frac{(n_2p_{max})^2}{n_2^{3/2}n_1^{-1/2}p_{max}^2},\frac{n_2p_{max}}{n_2^{1/2}n_1^{-1/2}p_{max}}\right)}\\
    &\leq e^{-c\sqrt{n_1n_2}}.
\end{align*}  

\paragraph{Control of the term $ \sum_{l\in L_2}(\sum_{i \in \Lambda_\delta}A_{il}x_{i})^2$.}
Let us define the event\[ \calE_1' = \left\lbrace \sum_{l\in [n_2]}Y_l^2\indic_{\lbrace l\in L_2\rbrace} \geq C_1 n_2e^{-C\sqrt{\log n_1}}p_{max}\right\rbrace.\] Note that \[ Y_l^2\indic_{\lbrace l\in L_2\rbrace}\leq C\sqrt{\log n_1}\sqrt{\frac{n_2}{n_1}}p_{max}\] for all distinct indices $i_1,i_2, i_3, i_4\in [n_1]$ we have 
\begin{align*}
     \expec\left(A_{i_1l}A_{i_2l}A_{i_3l}A_{i_4l}\indic_{\lbrace l\in L_2\rbrace}\right)&= \prob\left(\lbrace A_{i_1l}A_{i_2l}A_{i_3l}A_{i_4l}=1\rbrace \cap \lbrace \sum_{i\notin \lbrace i_1,\ldots, i_4\rbrace}A_{il}\geq M-4\rbrace\right)\\
     &\leq p_{max}^4e^{-C\sqrt{\log n_1}} \tag{by independence}.
 \end{align*}
 By consequence, using the same calculation as in \eqref{eq:a1}  we have $\expec(Y_l^2\indic_{\lbrace l\in L_2\rbrace})\lesssim p_{max}e^{-C\sqrt{\log n_1}}$ and $\expec(Y_l^4\indic_{\lbrace l\in L_2\rbrace})\lesssim \sqrt{\frac{n_2}{n_1}}p_{max}^2e^{-C\sqrt{\log n_1}}$. Bernstein's inequality implies that \[ \prob(\calE_1') \leq e^{-c\sqrt{n_1n_2}e^{-2C\sqrt{\log n_1}}}.\] But it is easy to check that, under the lemma assumptions, $n_1 \ll \sqrt{n_1n_2}e^{-2C\sqrt{\log n_1}}$.

\paragraph{Control of $\expec_{\Lambda^c}S_\delta - \expec_{A}S_\delta $.}Let us define the event  \[
\calE_2 =\left\lbrace |\expec_{\Lambda^c}S_\delta - \expec_{A}S_\delta | \geq C_2\sqrt{n_1n_2}p_{max}\right\rbrace.
\]
We have for all $t>0$
\begin{align*}
    \log \expec(e^{\expec_{\Lambda^c}S_\delta - \expec_{A}S_\delta}) & =\sum_{i,j.l}\log(e^{tw_{ij}p_{jl}}p_{il}+1-p_{il})-tw_{ij}p_{jl}p_{il}\\
    &\leq \frac{n_2p_{max}^3}{2}t^2e^{t\sqrt{\frac{n_2}{n_1}}p_{max}^2}
\end{align*}
By using Chernoff bound and the choice $t=\frac{1}{p_{max}^2}\sqrt{\frac{n_1}{n_2}}$ we obtain $\prob(\calE_2)\leq e^{-cn_1}$ for $C_2$ large enough.

\paragraph*{Conclusion.} Let us define 
\[ \calE(\delta) =\left\lbrace  \max_{l\in \Lambda_\delta} \sum_{i\in [n_1]} A_{il}\leq C \sqrt{\log n_1}\right\rbrace.\]
Note that for all $\delta $, $\calE\subset \calE(\delta)$ and $\calE(\delta)$ only depends on $\Lambda_\delta$.
For any fixed $\delta$, we have for $t=\sqrt{\frac{n_1}{n_2p^2_{max}}}$ and $C_2>0$ large enough \begin{align*}
    \prob &(\lbrace S_\delta -\expec_AS_\delta \geq 2C_2\sqrt{n_1n_2}p_{max}\rbrace \cap \calE)\\
    &= \expec(\prob(\lbrace S_\delta -\expec_AS_\delta \geq 2C_2\sqrt{n_1n_2}p_{max}\rbrace \cap \calE|\Lambda_\delta))\\
    &\leq  \expec(\indic_{\calE(\delta)}\prob(\lbrace S_\delta -\expec_AS_\delta \geq 2C_2\sqrt{n_1n_2}p_{max}\rbrace |\Lambda_\delta))\\
    &\leq \expec(\indic_{\calE(\delta)\cap\calE_1^c\cap\calE_1'^c\cap \calE_2^c}\prob(S_\delta -\expec_AS_\delta \geq 2C_2 \sqrt{n_1n_2}p_{max}|\Lambda_\delta))+3e^{-cn_1}\\
    & \leq  \expec(\indic_{\calE(\delta)\cap\calE_1^c\cap\calE_1'^c\cap \calE_2^c}\prob(S_\delta -\expec_{\Lambda^c}S_\delta \geq C_2\sqrt{n_1n_2}p_{max}|\Lambda_\delta))+3e^{-cn_1}\\
    &\leq \expec(\indic_{\calE_1^c}e^{p_{max}t^2\frac{e^{tM\sqrt{\frac{n_2}{n_1}}p_{max}}}{2}\sum_{l\in L_1}(\sum_{i \in \Lambda_\delta}A_{il}x_i)^2})e^{-tC_2\sqrt{n_2n_1}p_{max}}\\
    &+\expec(\indic_{\calE_1'^c}e^{p_{max}t^2\frac{e^{tC\sqrt{\log n_1}\sqrt{\frac{n_2}{n_1}}p_{max}}}{2}\sum_{l\in L_2}(\sum_{i \in \Lambda_\delta}A_{il}x_i)^2})e^{-tC_2\sqrt{n_2n_1}p_{max}}+3e^{-cn_1}\\
    &\leq 2e^{0.5n_2p_{max}^2t^2\frac{e^{t\sqrt{\frac{n_2}{n_1}}p_{max}}}{2}-tC_2\sqrt{n_2n_1}p_{max}}\\
    &\lesssim e^{-cn_1}.
\end{align*}
 By a union bound we have \[ \prob\left(\exists \delta, \calE \cap \lbrace S_\delta -\expec_AS_\delta \leq 2C_2\sqrt{n_1n_2}p_{max} \rbrace\right)\lesssim 2^{n_1}e^{-cn_1}\lesssim e^{-c'n_1} \] for a constant $c'>0$. It follows that, conditionned on $\calE$, with probability at least $1-e^{-c'n_1}$\[
S-\expec(S)=4(\expec_\delta(S_\delta)-\expec_\delta\expec_{A}S_\delta)\leq 8C_2\sqrt{n_1n_2}p_{max}.
\]
The stated result of the Lemma follows by symmetry of $S$ (the weights $w_{ij}$ can be negative). Note that the value of $c'$ depends only on the constants in the events we conditioned on. So, by choosing such constants large enough, we obtain $c'>11$.
\end{proof}

\begin{lemma}\label{lem:conc_prod_bin_gen2} Assume that the assumption of Theorem \ref{thm:improved_conc} are satisfied. Let, $I, J\subset [n_1]$ with $J\subset I$, $S=\sum_{i,j\in [n_1]}w_{ij}\langle A_{i:},A_{j:} \rangle$ where $w_{ii}=0$ for all $i$, $w_{ij}=\indic_{i\in I}\indic_{j\in J}$. Then for $C>0$ large enough we have
\[ \prob\left(\calE\cap \lbrace S\geq C|I||J|n_2p_{max}^2 \rbrace \right)\leq e^{-\frac{C}{2}n_2p_{max}^2C|I||J|}.\]
\end{lemma}
\begin{proof} 
Following the same calculation as in Lemma \ref{lem:conc_prod_bin_gen} one can show that 
\begin{align*}
    \log \expec_{\Lambda^c} &\left( e^{t(S_\delta-\expec_{\Lambda^c}(S_\delta))} \right) \leq \\
    &\leq  p_{max}t^2\sum_{j \in \Lambda_\delta^c\cap J}\sum_{l\in [n_2]}\frac{e^{t\norm{A_{:l}}_1}}{2}(\sum_{i \in \Lambda_\delta\cap I}A_{il})^2 \tag{by Taylor-Lagrange formula}\\
    &\leq  p_{max}t^2|J|\sum_{l\in [n_2]}\frac{e^{t\norm{A_{:l}}_1}}{2}(\sum_{i \in \Lambda_\delta\cap I}A_{il})^2\\
    &\leq p_{max}t^2|J|\sum_{l\in L_1}\frac{e^{tM}}{2}(\sum_{i \in \Lambda_\delta\cap I}A_{il})^2+p_{max}t^2|J|\sum_{l\in L_2}\frac{e^{tC\sqrt{\log n_1}}}{2}(\sum_{i \in \Lambda_\delta\cap I}A_{il})^2
    .
\end{align*}

We only highlight the main modifications since the proof is similar to Lemma \ref{lem:conc_prod_bin_gen}. 

\paragraph{Control of the term $ \sum_{l\in L_1}(\sum_{i \in \Lambda_\delta \cap I}A_{il})^2$.}
Let us define the event \[ \tilde{\calE}_1 = \left\lbrace \sum_{l\in [n_2]}\indic_{l\in L_1}(\sum_{i \in \Lambda_\delta \cap I}A_{il})^2 \geq C_1 n_2p_{max}\right\rbrace\] for a constant $C_1>0$ large enough. 

By the same calculation as in Lemma \ref{lem:conc_prod_bin_gen}, one can show that Bernstein inequality leads to \[ \prob\left( \sum_{l\in L_1}(\sum_{i \in \Lambda_\delta \cap I}A_{il})^2 \gtrsim n_2p_{max}|I|\right)\leq e^{-\Omega(n_2p_{max}|I|)}\leq e^{-\Omega(n_2p_{max}^2|I||J|)}.\]

\paragraph{Control of the term $ \sum_{l\in L_2}(\sum_{i \in \Lambda_\delta \cap I}A_{il})^2$.}

Let us define the event\[ \tilde{\calE}_1' = \left\lbrace \sum_{l\in [n_2]}Y_l^2\indic_{\lbrace l\in L_2\rbrace} \geq C_1 n_2e^{-C\sqrt{\log n_1}}p_{max}\right\rbrace.\]

By the same calculation as in Lemma \ref{lem:conc_prod_bin_gen}, one can show that Bernstein inequality leads to \[ \prob\left( \sum_{l\in L_2}(\sum_{i \in \Lambda_\delta \cap I}A_{il})^2 \gtrsim n_2p_{max}|I|e^{-C\sqrt{\log n_1}}\right)\leq e^{-\Omega(n_2p_{max}|I|)}\leq e^{-\Omega(n_2p_{max}^2|I||J|)}.\]

\paragraph{Conclusion.} We have shown that conditionally on $\calE $ \[ \log \expec_{\Lambda^c} \left( e^{t(S_\delta-\expec_{\Lambda^c}(S_\delta))} \right) \lesssim n_2p^2_{max}|I||J|t^2 e^{tM}+n_2p^2_{max}|I||J|t^2 e^{(t-c')C\sqrt{\log n_1}}.\] One can conclude by using Chernoff bound and choosing $t=1$.
\end{proof}


\end{document}